\documentclass[a4paper, 12pt]{article}

\usepackage[T2A]{fontenc}
\usepackage[utf8]{inputenc}
\usepackage[english]{babel}
\usepackage{amsfonts}
\usepackage{lipsum}
\usepackage{authblk}
\usepackage{amsmath,amssymb,amsthm}
\usepackage{graphicx}
\usepackage{abstract}
\usepackage{subcaption}
\usepackage{wrapfig}
\usepackage[hmargin=2cm,vmargin=2cm]{geometry}

\newtheorem{prop}{Proposition}
\newtheorem{theo}{Theorem}
\newtheorem{hypo}{Conjecture}

\newtheorem{defin}{Definition}
\newtheorem{corollary}{Corollary}
\newtheorem{lemma}{Lemma}

\begin{document}

\title{On the~First Caustic of Elliptical Billiards}
\author[1, 2]{Aleksandra Uskova}
\affil[1]{National Research University Higher School of Economics, Moscow, Russia}

\maketitle

\thispagestyle{empty}
\setlength{\parindent}{0pt}
\setlength{\parskip}{5pt}

\begin{abstract}
    \noindent A~point source of light is placed inside a~billiard with a~smooth, convex, closed boundary. For any integer $n$, the~$n$-th caustic by reflection, denoted by $\Gamma_n$, is the~envelope of light rays that have undergone $n$ reflections in such a~billiard after emanating from the~source. 

    \noindent It has been conjectured by Gil Bor and Serge Tabachnikov that for an~elliptical billiard, $\Gamma_n$ has exactly four ordinary cusps; this problem is a~billiard variation of Jacobi's Last Geometric Statement, which concerns the~number of cusps in the~conjugate locus of a~point on an~ellipsoid. Gil Bor, Mark Spivakovsky, and Serge Tabachnikov have proven that $\Gamma_n$ has at least four ordinary cusps. In this paper, we present a~proof that $\Gamma_1$ has exactly four ordinary cusps, using billiards in complex spaces.

\end{abstract}

\tableofcontents


\section{Introduction}

The~theorem, known as Jacobi's Last Geometric Statement or Jacobi's conjecture, states that any caustic (the~locus of the~first conjugate points) from any point on an~ellipsoid other than umbilical points has exactly four cusps. The~statement was rigorously proven only at~the~beginning of this century (see \cite{Jacobi}), and is still an~object of interest to researchers. In~particular, whether Jacobi's hypothesis is correct for the~$n$-th caustic (geometric locations of the~$n$-th conjugate points) for $n > 1$ remains open.

Around 2021, Serge Tabachnikov and Gil Bor reformulated the~problem in the~setting of billiards. In their article \cite{Bor-Tab}, they established the~following result.

\begin{defin}
    Let $\gamma$ be an~oval (a~smooth convex closed curve in $\mathbb{R}^2$) and an~ideal mirror; let~$O$ be a~point inside $\gamma$ and a~light source. For any $n\in\mathbb{N}$ \textbf{the~$n$-th caustic by reflection}~$\Gamma_n$ is the~envelope of the~family of rays emanating from $O$ that have undergone $n$ reflections in~$\gamma$.
\end{defin}

\begin{theo}[Bor, Tabachnikov]
    For each oval $\gamma\subset\mathbb{R}^2$ with a~general light source inside $\gamma$ and $n\geqslant 1$, the~$n$-th caustic by reflection $\Gamma_n\subset\mathbb{R}\mathrm{P}^2$ has at least four (ordinary) cusps.
\end{theo}

In the~same work, a~natural billiard analogue of Jacobi's conjecture was also formulated.

\begin{hypo}[Bor, Tabachnikov]
    If $\gamma$ is an~ellipse, then the~$n$-th caustic by reflection $\Gamma_n$ for a~light source inside $\gamma$ and different from a~focus has exactly four (ordinary) cusps for every $n \ge 1$.
\end{hypo}

In 2024, Gil Bor, Serge Tabachnikov, and Mark Spivakovsky made progress toward solving the~problem for elliptical billiards. In their paper \cite{Bor-Sp-Tab}, they were able to strengthen the~previous result and determine the~exact locations of the~four cusps.

\begin{theo}[Bor, Spivakovsky, Tabachnikov]
    Let $O$ be a~non-focal point inside an~ellipse $C$, and let $E$ and $H$ be the~ellipse and hyperbola~(respectively), passing through $O$ and confocal to~$C$. Consider the~four rays emanating from $O$ and tangent to $E$ and $H$ (two per each). Then, after $n$ reflections, the~four rays are tangent to $E$ and $H$ at four points which are cusps of the~$n$-th caustic by reflection from $O$.
\end{theo}

Thus, proving Conjecture 1 reduces to showing that the~four cusps described in Theorem 2 are the~only cusps of the~$n$-th caustic by reflection from $O$ and that all four cusps are ordinary. 

The~fact that the~first caustic by reflection in this setting has exactly four cusps was proven in~the~works of Bruce, Giblin, and Gibson (see, for example, \cite{BrGibGib}). In this paper, we also study the~characteristics of the~first caustic by reflection, including the~number of cusps, relying on techniques from algebraic geometry rather than methods from mathematical analysis and differential geometry. In order to be able to use the~tools of algebraic geometry, we consider a~complexification of the~problem. The~main results are Theorem 3 and Corollary 1.

\begin{theo}
    For every ellipse $\gamma~\subset \mathbb{C}\mathrm{P}^2$ other than a~circle and a~light source $O \in \mathbb{C}^2$ outside the~isotropic tangents to $\gamma$ and $\gamma$ itself, the~first caustic by reflection $\Gamma_1 \subset \mathbb{C}\mathrm{P}^2$ is an~algebraic curve with the~following values of the~numerical invariants: genus $g = 0$, degree $d = 6$, dual degree $d^* = 6$, number of ordinary cusps $\kappa~= 4$, number of inflection points $\kappa^*=4$, number of~double points $\delta~= 6$, number of bitangents $\delta^* = 6$.
\end{theo}

\begin{corollary}
    For every ellipse $\gamma~\subset \mathbb{R}^2$ and a~light source inside $\gamma$ different from its focus, the~first caustic by reflection $\Gamma_1 \subset \mathbb{R}\mathrm{P}^2$ has exactly four ordinary cusps.
\end{corollary}

The~proofs, key definitions, and technical details are provided in the~following sections. In~particular, Section~2 introduces the~tools that were used in the~main proofs, while Section~3 presents the~proof of Theorem~3. Section~4 contains a brief study of the first caustic by reflection in a real setting, in particular, the proof of Corollary~1. Section~5 is dedicated to a~small theorem about parallelogram billiard orbits in ellipses which was obtained during the~proof of Theorem~3:

\begin{theo}
    Consider a~complex 4-periodic orbit $(A,B,C,D)$ in the~non-circular ellipse \\ $\gamma: \cfrac{x^2}{a^2} + \cfrac{y^2}{b^2} = z^2$ in $\mathbb{C}\mathrm{P}^2$. $ABCD$ is a~parallelogram iff $ABCD$ is circumscribed \\ about $\gamma_\lambda: \cfrac{x^2}{a^2 + \lambda} + \cfrac{y^2}{b^2 + \lambda} = z^2$ with $\lambda~= -\cfrac{a^2b^2}{a^2 + b^2}$.
\end{theo}


\section{Preliminaries}

\subsection{Plücker formulas}

Consider a~curve $\mathcal{C}$ in the~complex projective plane defined by a~non-degenerate algebraic equation. Lines in the~plane can be identified with points of the~dual projective plane. The~tangents to $\mathcal{C}$ form the~dual curve on the~dual plane, which is traditionally denoted by $\mathcal{C}^*$.

Each such curve is characterised by four numerical invariants. 
\begin{itemize}
    \item $g$ is the~genus of both $\mathcal{C}$ and $\mathcal{C}^*$.
    \item $d$ and $d^*$ are the~degrees of $\mathcal{C}$ and $\mathcal{C}^*$, respectively. Geometrically, these invariants measure the~number of points where a~generic complex projective line intersects the~curve.
    \item $\delta$ and $\delta^*$ are the~numbers of ordinary double points of $\mathcal{C}$ and $\mathcal{C}^*$, i.e. points where the~curve has exactly two distinct tangents. In the~original plane, the~ordinary double points of the~dual curve correspond to bitangents to the~curve.
    \item $\kappa$ and $\kappa^*$ are the~numbers of ordinary cusps (also called 2-cusps) of $\mathcal{C}$ and $\mathcal{C}^*$, that is, points around which there is a~biholomorphic diffeomorphism mapping a~neighbourhood of the~point to a~neighbourhood of the~origin on the~(x, y)-plane, such that the~cusp is~sent to the~origin and the~curve is transformed into $y^2 = x^3$.
\end{itemize}

If $\mathcal{C}$ or $\mathcal{C}^*$ has higher order singularities, they are counted as multiple double points or cusps according to the~nature of the~singularity. For example, an~ordinary triple point is equivalent to three ordinary double points.

The~classical result by~Plücker states that the~above-mentioned seven characteristics of the~curve and its dual are linked by~simple algebraic relations.

\begin{theo}[Plücker formulas]
    For~numerical invariants of a~curve~$\mathcal{C}$ on~the~complex projective plane and its dual curve~$\mathcal{C}^*$, the~following equations are true:
    \begin{itemize}
        \item $d^*=d(d-1)-2\delta~-3\kappa$ and its dual version $d=d^*(d^*-1)-2\delta~^*-3\kappa^*$
        \item $\kappa^*=3d(d-2)-6\delta~-8\kappa$ and its dual version $\kappa~=3d^*(d^*-2)-6\delta^*-8\kappa^*$
        \item $g=\cfrac{1}{2}(d-1)(d-2)-\delta~-\kappa$ and its dual version $g=\cfrac{1}{2}(d^*-1)(d^*-2)-\delta^*-\kappa^*$
    \end{itemize}
\end{theo}

Knowing the~values of any three invariants allows one to uniquely determine the~remaining four via~the~Plücker formulas. Therefore, the~strategy of determining the~number of cusps of~caustics by reflection through finding the~values of their other characteristics is a~natural and promising path. However, to~use this powerful tool, the~complexification of the~entire problem is required.

\subsection{Complex billiards: the~new law of reflection and its properties}

We begin the~complexification of the~problem by expanding the~billiard system itself to the~complex domain. This subsection follows the~approach developed by Alexey Glutsyuk, see~\cite{Gl1}, \cite{Gl2}.

Consider the~complexification of $\mathbb{R}^2$ with the~Euclidean metric, namely, the~complex affine plane $\mathbb{C}^2\subset\mathbb{C}\mathrm{P}^2$ with the~complex-bilinear quadratic form $dz_1^2+dz_2^2$. We identify it with the~affine chart in $\mathbb{C}\mathrm{P}^2$ consisting of points $[z_1:z_2:1]$. The~line at infinity $\hat{\mathbb{C}}_\infty$ consists of points $[z_1:z_2:0] \in \mathbb{C}\mathrm{P}^2$. Further in this text we call a~line \textbf{finite} if it differs from the~line at~infinity.

The~points $I_1 = [1:-i:0]$ and $I_2 = [1:i:0]$ are called \textbf{isotropic points at infinity}. These points are distinguished by the~fact that they correspond to the~directions of straight lines on which the~form $dz_1^2 + dz_2^2$ vanishes (which basically means, that the~direction vector of these lines is orthogonal to itself). \textbf{Isotropic lines} are lines passing through at least one of the~isotropic points. In particular, the~line at infinity is isotropic. 

The~complexification of a~real ellipse $\gamma$ has four \textbf{isotropic tangents}, i.e, tangents that are isotropic lines. Note that the~line at infinity is never tangent to the~ellipse, since isotropic points do not belong to the~complexification of an~ellipse unless it is a~circle. In the~latter case, the~line at infinity intersects the~circle at both isotropic points and is, subsequently, not a~tangent. The~points of contact of the~isotropic tangents are called \textbf{points of isotropic contact}.

The~foci of a~complexified ellipse are defined as the~intersection points of non-parallel isotropic tangents (in other words, intersections of isotropic tangents distinct from the~isotropic points themselves). The~foci form two pairs such that any two foci lying on the~same isotropic tangent belong to different pairs. One of these pairs corresponds to the~standard foci of the~real ellipse.

The~most intriguing part is the~complexification of the~reflection law.

\begin{defin}
    The~reflection (symmetry) $\mathbb{C}^2\rightarrow\mathbb{C}^2$ with respect to a~non-isotropic complex line $L\subset\mathbb{C}\mathrm{P}^2$ is the~unique non-trivial involution such that it is a~complex isometry of the~form $dz_1^2+dz_2^2$ and fixes all points of~$L$. Accordingly, the~reflection (symmetry) at the~point $x\in L$ will be the~restriction of symmetry with respect to $L$ to the~space of lines passing through $x$.
\end{defin}

The~reflection in an~isotropic line at a~given point is defined as the~limit of symmetries with respect to non-isotropic lines.

\begin{defin}
    The~lines $l_1$ and $l_2$ passing through the~point $x$ are symmetric with respect to~the~isotropic line $L$ passing through the~point $x$ if there are lines $l_1^n\rightarrow~l_1$, $l_2^n\rightarrow l_2$, non-isotropic lines $L_n\rightarrow L$, and points $x_n\rightarrow x$, $x_n \in l_1^n, x_n \in l_2^n, x_n \in L_n$, such that for each $n\in\mathbb{N} $ $l_1^n$ and $l_2^n$ are symmetric with respect to~$L_n$ at~the~point $x_n$. 
\end{defin}

The~following propositions from \cite{Gl1}, \cite{Gl2} are very useful for understanding the~complex reflection law at the~points of an~ellipse. We present them here without proofs.

For convenience, let us introduce the~affine coordinate $z$ on the~line at infinity $\hat{\mathbb{C}}_\infty$ such that the~isotropic points at~infinity correspond to 0 and $\infty$.

\begin{prop}[Glutsyuk]
    Symmetry with respect to a~finite non-isotropic line $L$ intersecting~$\hat{\mathbb{C}}_\infty$ at~the~point $\varepsilon\in \hat{\mathbb{C}}_\infty\backslash\{0,\infty\}$ acts on~$\hat{\mathbb{C}}_\infty$ by~mapping $z\mapsto\cfrac{\varepsilon^2}{z}$.
\end{prop}

\begin{prop}[Glutsyuk]
    Let $L$ be an~isotropic line through a~finite point $x$. Two lines are symmetric with respect to $L$ iff some of them coincide with $L$.
\end{prop}

Propositions 1 and 2 imply that the~reflection map for elliptical non-circular billiards is uniquely defined at all points except for the~points of isotropic contact, where a~singularity arises when one attempts to reflect an~isotropic tangent in itself. Indeed, in the~case of reflection at a~finite point other than a~point of isotropic contact, one can uniquely determine the~image of any given line: it passes through the~point of reflection and the~point at $\hat{\mathbb{C}}_\infty$ given by Proposition~1. Tangents in infinite points of the~non-circular ellipse are non-isotropic, and we get the~image as a~limit of symmetries in finite points. But in points of isotropic contact, by Proposition~2, the~isotropic tangent is symmetric to any line passing through the~point of isotropic contact, with respect to itself; therefore, the~map is not well-defined.

For our purposes, it is important to understand how the~complex reflection law is connected to the~real one. Suppose that the~light source lies on the~real plane. Since isotropic points at infinity are not real, any restriction of a~complex billiard to the~real plane does not have singularities. Moreover, it coincides with the~usual real reflection if viewed as a~map between lines on the~plane rather than a~map between rays. It happens due to the~absence of~the~concepts of~'inside' and 'outside' in~the~complex domain. Consequently, if one tries to follow the~trajectory of a~point in the~restriction of a~complex billiard within a~closed curve of degree~$n$, after any reflection the~choice of the~next reflection point is non-unique as it can happen in~any other $(n-1)$ points where the~reflected line intersects the~reflecting curve. As we consider billiards in ellipses, which are curves of degree~2, such choice doesn't arise as after any reflection there is at~most one point, other than the~point of reflection, where the~reflected line intersects the~ellipse. Thus, we can conclude that for a~real light source a~complex billiard in an~ellipse reduces to a~real billiard within the~restriction of this ellipse to the~real domain with a~natural choice of the~direction of the~rays.

We refer to \cite{Rom}, \cite{Fier}, \cite{Fier-caustics} \cite{Wein1}, \cite{Wein2} as examples of obtaining excellent results using the~technique of complex billiards.

\subsection{Remarks on the~definition of the~n-th caustic by reflection}

Let us recall the~definition of the~$n$-th caustic by reflection given in \cite{Bor-Tab}.

\noindent\textbf{Definition 1.}
    \textit{Let $\gamma$ be an~oval (a~smooth convex closed curve in $\mathbb{R}^2$) and an~ideal mirror; let~$O$ be a~point inside $\gamma$ and a~light source. For any $n\in\mathbb{N}$ \textbf{the~$n$-th caustic by reflection}~$\Gamma_n$ is the~envelope of the~family of rays emanating from $O$ that have undergone $n$ reflections in~$\gamma$.}

To discuss the~complexification of the~problem at hand it is necessary to generalise the~definition of the~main object. Let us present the~modified definition.

\noindent\textbf{Definition 1'.}
    \textit{Let $\gamma$ be a~smooth compact curve in $\mathbb{C}\mathrm{P}^2$ and an~ideal mirror; let~$O$ be a~point on the~plane, such that $O \notin \gamma$, and a~light source. For any $n\in\mathbb{N}$ \textbf{the~(complex) $n$-th caustic by reflection}~$\Gamma_n$ is the~envelope of the~family of lines emanating from $O$ to all points of $\gamma$ that have undergone $n$ reflections in~$\gamma$ in the~sense of the~complex law of refection.}

Specifically, we extend the~setting to the~complex plane and consider complex curves as possible mirrors. Also, as the~complex domain lacks the~distinction between the~interior and the~exterior of the~mirror, we remove this condition.

If the~restriction of $\gamma$ to the~real plane is an oval, and the~light source $O$ is real and lies inside it, Definition~1' reduces to Definition~1. It can also occur that $\gamma$ and $O$ are real, but $O$ lies outside $\gamma$. To obtain the~real first caustic in this case, one can imagine that initial rays reflect at both points of intersection with $\gamma$, and the~reflected rays propagate to its interior. Though we do not focus on the~setting with the~external light source in this article, the~idea~is not new. E.g. Cayley studied the~first caustic by reflection in a~circle $\gamma$ for the~light source outside it, see $\cite{Cay}$.


\section{Proofs of Theorem 3 and Corollary 1}

The~idea~of the~proof of Theorem 3 is to find a~few values of the~numerical invariants of the~first caustic by reflection (in the~sense of Definition 1') and apply the~Plücker formulas afterwards. Thus, we present a~chain of small propositions before the~final proof itself.

\subsection{Genus of the~n-th caustics by reflection}

First of all, it is quite easy to fix the~genus. 

\begin{prop}
    For each ellipse $\gamma~\subset \mathbb{C}\mathrm{P}^2$ and a~light source $O\in \mathbb{C}\mathrm{P}^2$ not located on the~isotropic tangents to the~ellipse or the~ellipse itself $\Gamma_n$ has genus zero for each $n \in \mathbb{N}, n > 0$.
\end{prop}
\begin{proof}
    To each point $a~\in \gamma$ we associate the~corresponding point of the~caustic $C_a~\in \Gamma_n$; i.e.~the~$n$-th reflection of the~ray $OA$ is tangent to $\Gamma_n$ at $C_A$ and the~$n$-th reflections of close to~$OA$ rays focus at $C_A$. By this construction $\gamma$ is a~ramified covering of $\Gamma_n$. Thus, the~Riemann-Hurwitz formula~gives:

    $$0 = 2g(\gamma) = 2 + 2N(g(\Gamma_n) - 1) + \sum_{P \in \gamma}(e_P-1),$$

    where $N$ is the~degree of the~covering and $e_P$ is the~ramification index at $P$.

    As $2 + \sum_{P \in \gamma}(e_P-1) > 0$, the~only option for $g(\Gamma_n)$ is to be equal to zero.
\end{proof}

Let us note that the~correspondence described in the~proof is actually 1-to-1 correspondence up to self-intersections, as it is important in later subsections. 

Indeed, as a~tangent at any not self-intersection point $C \in \Gamma_n$ intersects $\gamma$ at most twice, there are no more than two points $a~\in \gamma$ such that $C = C_A$. Suppose that the~correspondence is two-to-one. Then for any point $a~\in \gamma$ there exists another point $A' \in \gamma$ such that $C_a~= C_{A'}$. Consequently, the~$n$-th reflection of $OA$ is a~tangent to $\Gamma_n$ at $C_A= C_{A'}$, and the~$n$-th reflection of this tangent is $OA'$. Therefore, $\Gamma_{2n}$ is exactly $O$ as all rays emanating from $O$ after $2n$~reflections pass through this point. But in such a~case we have a~contradiction with Theorem 1, which says that $\Gamma_{2n}$ has at least four ordinary cusps.

\subsection{Freedom of the~light source}

We continue with a~lemma~that allows us to disregard the~specific position of the~light source.

\begin{lemma}
    Consider an~elliptical billiard in $\gamma~\subset \mathbb{C}\mathrm{P}^2$ and light sources $O$ and $O'$. Let $\Gamma_1$ and $\Gamma_1'$ be the~first caustics by reflection corresponding to these light sources. If $O$ and $O'$ do not lie on the~isotropic tangents to $\gamma$, or $\gamma$ itself, then the~numerical invariants of the~curves $\Gamma_1$ and $\Gamma_1'$ appearing in the~Plücker formulas coincide.
\end{lemma}
\begin{proof}
    Note that points $O$ and $O'$ that do not lie on isotropic tangents or $\gamma$ can always be connected by a~path that does not intersect isotropic tangents. Indeed, $\mathbb{C}\mathrm{P}^2$ is a~path-connected space, and the~isotropic tangents and the~ellipse form five subspaces of (real) codimension two, so removing them does not affect the~path-connectedness of the~space.

    Therefore, as we move along a~path from $O$ to $O'$, the~first caustics by reflection remain well-defined for all source points on the~path, since no family of outgoing rays includes an~isotropic tangent to the~ellipse.

    Both $\Gamma_1$ and $\Gamma_1'$ are of genus zero by Proposition 3. 
    
    A~small shift of the~light source results in a~small perturbation of points of the~caustic, therefore, the~degree of the~caustic stays the~same as it represents the~number of intersections of the~caustic with a~generic line. 

    The~dual degree also remains the~same after a~small shift of the~light source. Let us fix the~light source $\hat{O}$ and some close point $\hat{O'}$ outside the~isotropic tangents to $\gamma$ and $\gamma$. As the~dual degree $\hat{d^*}$ represents the~number of tangents to the~caustic passing through a~generic point, and tangents to the~caustic are reflections of rays issued from $\hat{O}$, the~number of $a~\in \gamma$ such that $\hat{O}A$ reflects at $A$ to $A\hat{O'}$ equals $\hat{d^*}$. Therefore, if we now take $\hat{O'}$ as a~light source, the~dual degree of the~corresponding caustic is exactly $\hat{d^*}$ as there are $\hat{d^*}$ reflections of rays emanating from $\hat{O'}$ passing through $\hat{O}$.
    
    By the~Plücker formulas, the~other four numerical characteristics do not change either.
\end{proof}

Lemma~1 makes it possible to prove statements concerning the~first caustic by reflection by~choosing a~convenient position of the~light source. The~proof of Theorem 3 is thus correct for light sources not located on the~isotropic tangents to the~ellipse or on the ellipse itself, and this very condition defines the~position of the~light sources in Theorem 3 and in the~propositions of this section.

\subsection{Degree of the~curve dual to the~first caustic by reflection}

\begin{prop}
    For each ellipse $\gamma~\subset \mathbb{C}\mathrm{P}^2$ other than a~circle and a~light source $O\in \mathbb{C}\mathrm{P}^2$ not located on the~isotropic tangents to the~ellipse or the~ellipse itself the~'dual' degree of $\Gamma_1$ is~$d^* = 6$.
\end{prop}
\begin{proof}
    By Lemma~1, we can consider any particular light source not lying on the~isotropic tangents or the~ellipse since the~numerical invariants remain constant for all other such sources.

    Let us choose a~light source $O$ on the~line at infinity located at a~point other than the~isotropic one and such that $O \notin \gamma$. Let us introduce an~affine coordinate $z$ on the~line at infinity such that the~isotropic points correspond to $z = 0$ and $z = \infty$. Suppose that the~coordinate of $O$ equals $\varepsilon$.

    Recall that $d^*$ is the~degree of the~curve dual to $\Gamma_1$, i.e. the~number of intersection points with the~generic line (counted with multiplicities); in terms of the~original curve $\Gamma_1$, $d^*$ is the~number of tangents to $\Gamma_1$ passing through the~point of general position (again, counted with multiplicities). Tangents to $\Gamma_1$ are the~first reflection lines, since $\Gamma_1$ is, by definition, an~envelope of the~first reflections. Thus, in order to find $d^*$, it suffices to determine the~number of reflected lines passing through a~generic point on the~plane.

    Consider some point $B$, other than an~isotropic one, on the~line at infinity. Let $B$ have the~coordinate $\beta$. Which rays emanating from $O$ can be reflected to a~line passing through $B$? 

    If the~initial ray intersects an~ellipse at the~point which does not belong to the~line at infinity, in order for the~reflected ray to pass through $B$, the~tangent to the~ellipse at the~point of reflection must intersect the~line at infinity either at point $A$ with the~coordinate $\sqrt{\varepsilon\beta}$, or at point $A'$ with the~coordinate $-\sqrt{\varepsilon\beta}$ by Proposition 1. There are four such points on the~ellipse as there are two tangents to the~ellipse through $A$ and two tangents through $A'$. Thus, we obtain four reflected rays passing through $B$.

    If the~initial ray belongs to the~line at infinity, and the~point of reflection is, consequently, infinite, let us look at the~limit of close reflections. As $\gamma$ is not a~circle, points of intersections of $\gamma$ with the~line at infinity are not isotropic. Suppose the~coordinate of any of these two points is $\alpha$, note that $\alpha~\neq \varepsilon$ as $O \notin \gamma$. We can construct a~sequence of reflections of lines from $O$ in close points of the~ellipse with tangents intersecting the~line at infinity at $\alpha~+ n^{-1}$. By Proposition 1 reflections in these points intersect the~line at infinity at $\cfrac{(\alpha~+ n^{-1})^2}{\varepsilon}$. Thus, the~limit of reflected lines as $n \to \infty$ passes through the~points at infinity $\alpha$ and $\cfrac{\alpha^2}{\varepsilon} \neq \alpha$, i.e. the~limit of reflections is the~line at infinity, which surely passes through $B$. It gives us two more reflections passing through $B$.

    Let us now discuss the~question of multiplicities. Let $B^*$ be the~line in the~dual space. If it is intersects $\Gamma_1^*$ at point $l^*$ with the~multiplicity higher than 1, then it either means that $B^*$ is at~least tangent to $\Gamma_1^*$ at $l^*$ or that $l^*$ is a~singular point. Note that these options can combine. The~first case corresponds to $B \in \Gamma_1$ in the~original space, with $l$ which is at least tangent to $\Gamma_1$ at $B$. The~second case corresponds to $B \in l$ with $l$ being a~bitangent or a~tangent at an~inflection point. Consequently, if we choose our point $B$ in such a~way that it does not belong to the~first caustic and does not correspond to the~directions of bitangents or tangents at inflection points, all multiplicities equal 1.

    Thus, we conclude that $d^* = 6$ for $\Gamma_1$.
\end{proof}

\subsection{The hunt for bitangents}

In order for the Plücker formulas to give an~accurate result, we need to evaluate at least one other invariant. Continuing the~previous subsection, we again study an~invariant related rather to a~curve dual to the~first caustic by reflection - we try to catch its~bitangents.

\begin{prop}
    For each ellipse $\gamma~\subset \mathbb{C}\mathrm{P}^2$ other than a~circle and a~light source $O\in \mathbb{C}\mathrm{P}^2$ not located on the~isotropic tangents to the~ellipse or the~ellipse itself the~number of bitangents of~$\Gamma_1$ is $\delta^* \le 6$.
\end{prop}
\begin{proof}
As in the~proof of Proposition 4, we choose a~light source $O$ on the~line at infinity in such a~way that it is not the~isotropic point or the~point of intersection of $\gamma$ with infinity.

We again use an~affine coordinate $z$ on the~line at infinity such that the~isotropic points correspond to $z = 0$ and $z = \infty$, and the~coordinate of $O$ equals $\epsilon$.

Our goal is to estimate the~number of bitangents to the~first caustic by reflection $\Gamma_1$. As we said before, the~caustic is an~envelope of the~first reflections, thus any tangent is the~first reflection of the~ray from $O$ at some point of $\gamma$. Therefore, bitangents correspond to situations where reflections at different points of $\gamma$ coincide. Let us investigate how it can happen. 

\vspace{5mm}

\textbf{Case I.} Consider two rays issued from $O$ such that they lie on the~same line $l$ which is not tangent to $\gamma$. To obtain a~bitangent, the~rays have to reflect into the~same line at the~two points of intersection of $l$ with $\gamma$. Consequently, both rays have to reflect into $l$ itself.

\vspace{4mm}

\hspace{4mm} \textbf{I.I.} If $l$ is the~line at infinity, both reflections are, indeed, the~line at infinity, see the~proof of~Proposition 4. 

\vspace{4mm}

\hspace{4mm} \textbf{I.II.} If $l$ is a~finite line, it has to be a~normal line at both points of intersections with $\gamma$, as it is the~only way to be reflected into itself being non-tangent. The~only such lines are the~axes of $\gamma$. We cannot be sure that $O$ lies at the~intersection of one of the~axes with infinity, but we can note that if it is the~case, then this particular bitangent passes through the centre of $\gamma$. 

\vspace{4mm}

\textbf{Case II.} Consider two rays issued from $O$ suсh that they do not lie on the~same line $l$.

Since we are concerned with bitangents, we can immediately exclude configurations that are obviously irrelevant.
\begin{itemize}
    \item If one of the~rays belongs to the~line at infinity, the~other does not, as the~rays cannot lie on the~same line. Therefore, one of the~reflections is the~line at infinity, while the~other is some finite line. They clearly do not coincide.
    \item If one of the~rays belongs to a~line tangent to $\gamma$, it is non-isotropic, and its reflection is this tangent line. It cannot coincide with the~reflection at any other point as it intersects $\gamma$ in the~only point.
    \item A~bitangent cannot pass through any point of isotropic contact as at them all rays reflect into isotropic tangents (the~latter is true as we choose $O$ outside the~isotropic tangents).
\end{itemize}

Thus, we can 'catch' a~bitangent only in cases where both rays belong to finite non-tangent lines such that for each line at least one point of its intersection with $\gamma$ is not a~point of isotropic tangency.

Let us denote the~points of intersection of the~first ray with $\gamma$ as $A$ and $B$, of the~second ray with $\gamma$ - as $C$ and $D$. As all points are finite, the~intersections of tangents at these points with the~line at infinity are well-defined. Let us denote the~coordinates of these intersections as $a$, $b$, $c$ and $d$.

Suppose that $BC$ is both the~reflection of $AB$ at $B$ and $DC$ at $C$, i.e. $BC$ is a~bitangent to the~first caustic by reflection. It means that $(\dots,A,B,C,D,\dots)$ is a~billiard trajectory at $\gamma$. Thus, there is the~confocal to $\gamma$ conic $\gamma_\lambda$ such that $AB$, $BC$ and $CD$ are all tangent to $\gamma_\lambda$. 

As $AB$ and $CD$ are parallel (they intersect at the~infinity) and are tangent to the~same conic, they are symmetric with respect to the~shared centre of $\gamma$ and $\gamma_\lambda$. Therefore, the~unordered pair of points $(A, B)$ is symmetric to the~unordered pair $(C, D)$.

\vspace{4mm}

\hspace{4mm} \textbf{II.I.} Suppose $B = -C$, $A = -D$. It also means that $b = c$, $a~= d$, as tangents at symmetric points are parallel.

In this case we again note that the~bitangent passes through the~centre of $\gamma$, as it passes through symmetric points $B$ and $C$. And also, even though we cannot know for sure whether $AD$ is a~bitangent or not, we surely know that if it is, it passes through the~centre of $\gamma$ too.

\vspace{4mm}

\hspace{4mm} \textbf{II.I.} Suppose $B = -D$, $A = -C$. It also means that $b = d$, $a~= c$.

This case is more intriguing. As $BC$ is a~reflection of both $AB$ and $CD$, we can say that $\cfrac{b^2}{\varepsilon} = \cfrac{c^2}{\varepsilon}$ by Proposition 1. $b \neq c$ because there are no more than two tangents with the~same direction to $\gamma$. Thus, we conclude that $d = b = -c = -a$.

But then at $A$ $AB$ is reflected into a~line $l_A$ which is tangent to $\gamma_\lambda$ and intersects the~line at infinity at $\cfrac{a^2}{\varepsilon}$. Similarly, $CD$ is reflected at $D$ into a~line $l_D$ which is tangent to $\gamma_\lambda$ and intersects the~line at infinity at $\cfrac{d^2}{\varepsilon}$. As $d = b = -c = -a$, we get that $BC$, $l_A$ and $l_D$ are parallel and all three are tangent to $\gamma_\lambda$. It can be possible only if $l_A = l_D$.

We conclude that in this setting both $BC$ and $AD$ are bitangents. Note that these bitangents do not pass through the~centre of $\gamma$.

\vspace{5mm}

Finally, we see that there are three possible types of bitangents:
\begin{enumerate}
    \item the~line at infinity;
    \item a~finite line passing through the~centre of $\gamma$;
    \item the~second pair of parallel sides of a~parallelogram inscribed in $\gamma$ and described near the~conic confocal to $\gamma$ with the~first pair of parallel sides directed to $O$.
\end{enumerate}

There is one line of the~first type. By Proposition 4, we have $d^* = 6$, and thus there are six reflections passing through the~centre of $\gamma$. Therefore, there are no more than three bitangents of the~second type.

It remains to show that there are no more than two bitangents of the~third type, i.e. there is at most one $\gamma_\lambda$ such that we can find a~parallelogram described above.

Without loss of generality, let us assume that $\gamma$ is given by the~equation $$\frac{x^2}{\alpha^2} + \frac{y^2}{\beta^2} = z^2$$

Any confocal conic $\gamma_\lambda$ corresponds to the~equation $$\frac{x^2}{\alpha^2+\lambda} + \frac{y^2}{\beta^2+\lambda} = z^2$$

Let the~point $O = [v_x:v_y:0]$. We normalise the~vector $(v_x, v_y)$ so that $v_x^2 + v_y^2 = 1$. This is possible, since $O$ is not isotropic. Any finite line passing through $O$ has the~equation $v_y\cdot x - v_x\cdot y + m \cdot z = 0$, where $m$ is some constant. 

Tangents to $\gamma_\lambda$ passing through $O$ are finite as the~line at infinity intersects $\gamma_\lambda$ in two points, so to find their equations it remains to find appropriate $m$. As the~only point at infinity is $O$ we can consider $z = 1$.

$$v_y \cdot x - v_x \cdot y + m = 0 \Rightarrow v_x \cdot y = v_y\cdot x  + m$$

Consider $v_x \neq 0$. The~line $v_x \cdot y = v_y\cdot x  + m$ is tangent to $\gamma_\lambda$ at a~finite point if and only if $\cfrac{v_x^2x^2}{\alpha^2+\lambda} + \cfrac{(v_yx+m)^2}{\beta^2+\lambda} = v_x^2$ has the~only solution. Let us set $\alpha^2+\lambda~= a_\lambda$, $\beta^2+\lambda~= b_\lambda$ and find a~discriminant:

$$ D = \left( \frac{2v_ym}{b_\lambda} \right)^2 - 4 \left( \frac{v_x^2}{a_\lambda} + \frac{v_y^2}{b_\lambda}\right) \left(\frac{m^2}{b_\lambda} - v_x^2 \right) = 4 \left(  \frac{v_x^4}{a_\lambda} + \frac{v_x^2v_y^2}{b_\lambda} - \frac{v_x^2m^2}{a_\lambda~b_\lambda}\right) = \frac{4v_x^2}{a_\lambda~b_\lambda} (v_x^2 b_\lambda~+ v_y^2 a_\lambda~- m^2)$$

Thus, as $v_x \neq 0$,

$$D = 0 \iff m^2 = v_x^2 b_\lambda + v_y^2 a_\lambda = v_x^2(\beta^2+\lambda) + v_y^2(\alpha^2+\lambda) = v_x^2 \beta^2 + v_y^2 \alpha^2+\lambda$$

In case $v_x = 0$, we get $O = [0:1:0]$, i.e. a~'vertical' direction. Thus, the~points of tangency are the~vertices of $\gamma_\lambda$, namely, $[\pm\sqrt{\alpha^2+\lambda}:0:1]$. Therefore, 
$$m^2 = \alpha^2+\lambda~= 0\cdot \beta^2 + 1 \cdot \alpha^2 + \lambda,$$

which is exactly the~same expression as in case $v_x \neq 0$.

We conclude that two tangents to $\gamma_\lambda$ from $O$ are given by the~equations

$$v_y \cdot x - v_x \cdot y \pm \sqrt{v_y^2 \alpha^2 + v_x^2 \beta^2 + \lambda}\cdot z = 0.$$

As $O \neq \gamma$ and $O$ is not an~isotropic point at infinity, one tangent intersects $\gamma$ at finite points $A$ and $B$, the~other one intersects it at finite points $C$ and $D$, and all four points are distinct. As $AB$ and $CD$ are symmetric with respect to the~centre of $\gamma$, the~points form two pairs of symmetric points. We choose the~names in such a~way that $C = -A$ and $D = -B$.

It is clear that $BC // AD$ because of the~symmetry, so our goal is to check whether $BC$ and $AD$ are tangent to $\gamma_\lambda$. If it is the~case, $ABCD$ is a~desired parallelogram, and $BC$ and $AD$ are bitangents to the~first caustic by reflection.

Let us denote 
$$l_{AB} = v_y \cdot x - v_x \cdot y + \sqrt{v_y^2 \alpha^2 + v_x^2 \beta^2 + \lambda}\cdot z,$$
$$l_{CD} = v_y \cdot x - v_x \cdot y -\sqrt{v_y^2 \alpha^2 + v_x^2 \beta^2 + \lambda}\cdot z,$$ 
and
$$Q_\gamma~= \frac{x^2}{\alpha^2}+ \frac{y^2}{\beta^2} - z^2.$$
Consider a~pencil of conics $\mathcal{C}(k): Q_\gamma~+ k(l_{AB}l_{CD}) = 0$ with $k \in \mathbb{C} \cup \{ \infty \}$. As $A, B, C, D$ are the~basic points of this pencil, $BC \cup AD$ is one of the~degenerate conics in the~pencil. Let us find the~corresponding value of $k$.

The~matrix of $\mathcal{C}(k)$ is given by

$$Q_{\mathcal{C}(k)} = 
\begin{pmatrix} 
\frac{1}{\alpha^2} + kv_y^2 & -k v_x v_y & 0 \\ 
-k v_x v_y & \frac{1}{\beta^2} + k v_x^2 & 0 \\ 
0 & 0 & -1 - k m^2 
\end{pmatrix}$$

where $m^2 = v_y^2 \alpha^2 + v_x^2 \beta^2 + \lambda$.

The~degenerate conics correspond to $\det Q_{\mathcal{C}(k)} = 0$. We calculate:

$$\det Q_{\mathcal{C}(k)} = (-1 - km^2) \left( \left(\frac{1}{\alpha^2} + kv_y^2\right) \left( \frac{1}{\beta^2} + kv_x^2 \right) - k^2v_x^2v_y^2\right) = \frac{1}{\alpha^2 \beta^2}(-1-km^2)(1 + k(v_y^2 \alpha^2 + v_x^2 \beta^2))$$

As we chose $O$ in such a~way that it is not the~intersection of $\gamma$ and the~line at infinity, \\ i.e. $O \neq [\alpha: \pm i\beta:0]$, we have $v_y^2 \alpha^2 + v_x^2 \beta^2 \neq 0$. 

Thus, $\det \mathcal{C}(k) =0$ with $k_1 = -\cfrac{1}{m^2} = -\cfrac{1}{v_y^2 \alpha^2 + v_x^2 \beta^2 +\lambda}$ or $k_2 = -\cfrac{1}{v_y^2 \alpha^2 + v_x^2 \beta^2}$. The~value~$k_1$ corresponds to the~degenerate conic without $z$-terms, i.e. to the~degenerate conic passing through $[0:0:1]$, which is $AC \cup BD$. Therefore, $BC \cup AD$ corresponds to $k_2 = -\cfrac{1}{v_y^2 \alpha^2 + v_x^2 \beta^2}$.

Let us find the~explicit equation for $BC \cup AD = \mathcal{C}(k_2)$:

$$\frac{x^2}{\alpha^2} + \frac{y^2}{\beta^2} - z^2 -\frac{1}{v_y^2 \alpha^2 + v_x^2 \beta^2} ((v_y  x - v_x  y)^2 - (v_y^2 \alpha^2 + v_x^2 \beta^2 +\lambda)z^2) = 0,$$

$$\frac{x^2}{\alpha^2} + \frac{y^2}{\beta^2} -\frac{1}{v_y^2 \alpha ^2 + v_x^2 \beta^2} ((v_y  x - v_x  y)^2 - \lambda z^2) = 0,$$

$$\frac{(v_y^2 \alpha^2 + v_x^2 \beta^2)x^2}{\alpha^2} + \frac{(v_y^2 \alpha^2 + v_x^2 \beta^2)y^2}{\beta^2} -(v_y  x - v_x  y)^2 + \lambda z^2 = 0,$$

$$\frac{v_x^2 \beta^2 x^2}{\alpha^2} + \frac{v_y^2 \alpha^2y^2}{\beta^2} + 2v_xv_y  xy + \lambda z^2 = \left(\frac{v_x\beta x}{\alpha} + \frac{v_y \alpha y}{\beta}\right)^2 + \lambda z^2= 0.$$

So all we need to do now is to check whether lines $\cfrac{v_x \beta}{\alpha} x+ \cfrac{v_y \alpha}{\beta} y\pm \sqrt{-\lambda} z = 0$ are tangent to~$\gamma_\lambda: \cfrac{x^2}{\alpha^2 +\lambda} + \cfrac{y^2}{\beta^2 + \lambda} - z^2 = 0$. As we have seen above, it is necessary that the~coefficients satisfy the~relation:

$$(\pm\sqrt{-\lambda}) ^2= \left(\cfrac{v_x \beta}{\alpha} \right)^2 \cdot (\alpha^2 + \lambda) + \left( \cfrac{v_y \alpha}{\beta} \right)^2 \cdot (\beta^2 + \lambda).$$

Let us simplify:

$$-\lambda~= \cfrac{v_x^2 \beta^2}{\alpha^2} \cdot (\alpha^2 + \lambda) +  \cfrac{v_y^2 \alpha^2}{\beta^2}  \cdot (\beta^2 + \lambda) = v_y^2 \alpha^2 + v_x^2 \beta^2 + \lambda \left( \cfrac{v_x^2 \beta^2}{\alpha^2} + \cfrac{v_y^2 \alpha^2}{\beta^2} \right),$$

\begin{align*}
-(v_y^2\alpha^2 + v_x^2\beta^2) 
&= \lambda~\left( \cfrac{v_x^2\beta^2}{\alpha^2} + \cfrac{v_y^2\alpha^2}{\beta^2} + 1\right) = \\
&=
\lambda~\left( \cfrac{v_x^2\beta^2}{\alpha ^2} + \cfrac{v_y^2\alpha^2}{\beta^2} + v_x^2 + v_y^2 \right) = \\
&= \lambda\left( \cfrac{v_x^2(\alpha^2+\beta^2)}{\alpha^2} + \cfrac{v_y^2(\alpha^2+\beta^2)}{\beta^2}\right) = \lambda~\cfrac{(v_y^2\alpha^2+v_x^2\beta^2)(\alpha^2+\beta^2)}{\alpha^2\beta ^2}.
\end{align*}

As we have noted above, $v_y^2\alpha^2+v_x^2\beta^2 \neq 0$, consequently we get the~only appropriate value of $\lambda$:

$$\lambda~= -\cfrac{\alpha^2\beta^2}{\alpha^2 + \beta^2}$$

Moreover, as this value does not depend on $O$, any parallelogram orbit in $\gamma$ is necessary circumscribed about $\gamma_\lambda$ with $\lambda~= -\cfrac{\alpha^2\beta^2}{\alpha^2 + \beta^2}$.

We conclude that there is exactly one $\gamma_\lambda$ for which we can find a~parallelogram orbit, and therefore there are exactly two bitangents of the~third type. It ends the~proof.

\end{proof}

\subsection{Final proof}

Now we are ready to present the~final proofs of Theorem 3.

\begin{proof}[\textbf{Proof of Theorem 3}]
    It is known from Propositions 3-5 that for an~elliptical billiard and a~light source not located on the~isotropic tangents to the~ellipse or the~ellipse itself, the~first caustic by reflection $\Gamma_1\subset \mathbb{C}\mathrm{P}^2$ has genus $g=0$ and the~'dual' degree $d^* =6$, and the~number of bitangents is $\delta^* \leqslant 6$.

    By Theorem 2, proven in \cite{Bor-Sp-Tab}, $\kappa\geqslant 4$ for $\Gamma_1$ in the~case of real light sources. Applying Lemma~1, we obtain that the~result holds in the~general complex case under consideration.

    Let us substitute the~obtained values into the~Plücker formulas:

    \begin{enumerate}
        \item $6=d(d-1)-2\delta~- 3\kappa$
        \item $d=6\cdot5-2\delta^*-3\kappa^*$
        \item $\kappa^*=3d(d-2)-6\delta~-8\kappa$ 
        \item $\kappa~=3\cdot6\cdot4-6\delta^*-8\kappa^*$
        \item $0=\cfrac{1}{2}(d-1)(d-2)-\delta~-\kappa$ 
        \item $0=\cfrac{1}{2}\cdot5\cdot4-\delta^*-\kappa^*$ 
    \end{enumerate}

    The~sixth equality implies that $\delta^*+\kappa^* = 10$. We substitute these values into the~second and fourth relations.

    \hspace{4mm} 2. $d=6\cdot5-2(\delta^*+\kappa^*)-\kappa^* = 30 - 20 - \kappa^* = 10 - \kappa^*;$
    
    \hspace{4mm} 4. $\kappa~=3\cdot6\cdot4-6\delta^*-8\kappa^* = 3\cdot6\cdot4 -6(\delta^*+\kappa^*)-2\kappa^* = 72 - 60 - 2\kappa^* = 12 - 2\kappa^*.$    

    As $\delta^* \leqslant 6$, it follows from the~equality $\delta^*+\kappa^* = 10$ that $\kappa^* \geqslant 4$. Consequently we get that $\kappa~\geqslant 4$, $\kappa^* \geqslant 4$ and $\kappa~+ 2\kappa^* = 12$. It is possible only if $\kappa~= \kappa^* = 4$.

    Thus, we immediately get $\delta^* = d = 10 - \kappa^* = 10-4 = 6$, and the~fifth equation gives $\delta~= 6$.
\end{proof}

\subsection{Ellipse vs. circle}

In the paper from 1857, Cayley obtained a lot of results about the first caustic by reflection in~a~circle \cite{Cay}. There is room for a~comparison between Cayley's results and ours. Contrary to intuition, the numerical invariants for the case where $\gamma$ is a circle are different from the ones obtained in Theorem 1.

The values from Cayley's 'Memoir' are as follows (see p. 290 of Cayley's paper): the caustic is a curve of degree $d = 6$ and dual degree $d^* = 4$; it has 4 double points, 3 bitangents, 6 cusps, and 0 points of inflection. 

The natural reason for such a difference is the basic fact that only conics intersecting the line at infinity at isotropic points are circles. Therefore, the reflection at the points at infinity of $\gamma$ differs dramatically depending on whether it is a circle or not. In fact, Cayley specifically notes that two of the six cusps of the caustic by reflection in a circle lie precisely at the isotropic points at infinity, and their tangents are different from the line at infinity. Thus, one can see, for example, how $d^* = 4$ arises: if we subtract the 'infinite' part in Proposition 4, we are left with precisely this value.


\section{First caustic by reflection as a curve in $\mathbb{R}\mathrm{P}^2$}

In this section we return to the initial real-world setting to see, what conclusions can be drawn about the first caustic by reflection in $\mathbb{R}\mathrm{P}^2$.

First of all, with Theorem 3, we can prove Corollary~1 and establish the number of real cusps.

\begin{proof}[\textbf{Proof of Corollary 1}]
    Suppose $\gamma$ is an~ellipse but not a~circle. The foci of the~ellipse are the~only intersections of isotropic tangents to the~ellipse and the~real plane. Therefore, by Theorem 3, there are four ordinary cusps on $\Gamma_1\subset \mathbb{C}\mathrm{P}^2$, but by Theorem 1 all these cusps belong to the~real part of the~caustic.

    In the case where $\gamma$ is a~circle, it was shown by Cayley in \cite{Cay} that there are exactly four real cusps on the~first caustic by reflection if the~light source lies inside $\gamma$ but not at its~centre.
\end{proof}

Secondly, we can try to study the caustic using some real analogues of Plücker formulas. 

There is a famous result by~Klein \cite{Kl}, linking the~degree and dual degree of a~curve with the~numbers of its special points.

\begin{theo}[Klein's formula]
    If a real algebraic curve $C$ has no singularities except nodes, cusps, bitangents, and inflection points, then
    
    $$d + 2\tau' + i = d^* + 2\Delta + u,$$
    
    where $d$ is the degree, $\tau'$ is the number of bitangents with non-real points of tangency, $i$ is the number of real points of inflection, $d^*$ is the class (i.e., the degree of the dual curve), $\Delta$ is the number of acnodes (isolated real double points), and $u$ is the number of real cusps of $C$.
\end{theo}

By Theorem 3, we know that $d = d^* = 6$, and by Corollary 1, we have $u = 4$. By the construction of the first caustic by reflection as the envelope of a family of lines, we get $\Delta = 0$ (there are no isolated points). We can also quickly establish the number of points of inflection.

\begin{lemma}
    The number of real points of inflection of the first caustic by reflection is $i = 0$.
\end{lemma}
\begin{proof}
    The~orthotomic curve of the~ellipse $\gamma$ with respect to the~light source $O$ is the~locus of the~reflections of $O$ in the~tangents of $\gamma$. It is a~smooth closed curve. 
    
    The first caustic by~reflection is the~evolute (i.e., the~envelope of the normal lines) of the~orthotomic curve, and the~evolute of any smooth plane curve has no 
    ~inflection points (one can find this fact, for example, in Section 4.1 of \cite{OvsTab}).
\end{proof}

Thus, Klein's formula gives:

$$6 + 2\tau' + 0 = 6 + 2 \cdot 0 + 4 \Rightarrow \tau' = 2.$$

Therefore, since the first caustic by reflection has six bitangents by Theorem 3, we can conclude that four of them have real points of tangency.

Interesting equations were also obtained by Fabricius-Bjerre \cite{Fab-Bj} in the 1960s--1970s. He considered curves in $\mathbb{R}^2$, and therefore his results cannot be used directly for the study of caustics. However, in 2006, Abigail Thompson generalized the Fabricius-Bjerre formula to singular curves in  $\mathbb{R}\mathrm{P}^2$ \cite{Thom}. Let us present it here.

Consider the~real projective plane $\mathbb{R}\mathrm{P}^2$ equipped with the~standard spherical metric, which it inherits from its double cover, the~unit 2-sphere~$S^2$. Note that the introduction of the~metric in Thompson's paper is crucial: in original Fabricius-Bjerre formulas, the~concept of lying in one half-plane is used; the metric allows one to construct some analogues in $\mathbb{R}\mathrm{P}^2$ and, thus, to~overcome the~fact that a~simple closed geodesic does not divide $\mathbb{R}\mathrm{P}^2$ into two distinct "sides". 

Consider a~generic, oriented, closed curve $C$ in $\mathbb{R}\mathrm{P}^2$. Let the~curve be smoothly immersed except for finitely many ordinary cusps — meaning that at each cusp the~two branches of the~curve locally lie on opposite sides of the~tangent geodesic. Then the~topological and geometric invariants of the~curve are related by~the~following equation:  

$$t_1 - t_2 = \delta_1 + \delta_2 + \frac{1}{2}i + u - \frac{1}{2}a_1 + \frac{1}{2}a_2.$$

In this formula:
\begin{itemize}
    \item $t_1$ and $t_2$ are the numbers of double-supporting geodesics of two types. A~double-supporting geodesic is a~projective line that is either tangent to~$C$ at~two points, tangent at one point and passing through a cusp, or passing through two cusps. The~two contact points divide the~geodesic into two segments; one of the segments is strictly shorter than $\pi/2$. If the~curve's branches lie on the same side of this segment, it is classified as the~first type ($t_1$). If they lie on opposite sides, it is the second type ($t_2$).  
    \item $\delta_1$ and $\delta_2$ are the~numbers of double points of two types. The~tangent geodesics at any crossing of $C$ define four angles, two of which, $\alpha$ and $\beta$ are less than $\pi / 2$ (because of a generality of C). In a~small neighbourhood of a~crossing, there are four segments of $C$. If one of the segments lies in $\alpha$ and another in $\beta$, then the double point is of the first type ($\delta_1$). If two segments both lie in $\alpha$ or both in $\beta$, it is the second type ($\delta_2$).
    \item $i$ is the number of points of inflection.  
    \item $u$ is the number of ordinary cusps on the curve.
    \item $a_1$ and $a_2$ are the~numbers of antipodal pairs of two types. An~ordered pair of points $(p, q)$ on the~curve is an~antipodal pair if $q$ lies on the~tangent geodesic to~$p$ at a~distance of exactly $\pi/2$. The~classification into the~first type ($a_1$) or the~second type ($a_2$) is determined by~how the~tangent geodesic at~$q$ is positioned relative to the~centre of curvature at~$p$.
\end{itemize}

For the~first caustic by reflection, the~formula reduces~to:

$$t_1 - t_2 = \delta_1 + \delta_2 + 0 + 4 - \frac{1}{2}a_1 + \frac{1}{2}a_2,$$

and can be further studied.


\section{Parallelogram billiard orbits in the~ellipse}

In this short section we discuss the~last part of the~proof of Proposition 5 as it is quite intriguing by itself.

As a~result of our calculations in this proof we can formulate an~independent theorem.

\textbf{Theorem 4}
\textit{Consider a~complex 4-periodic orbit $(A,B,C,D)$ in the~non-circular ellipse \\ $\gamma: \cfrac{x^2}{a^2} + \cfrac{y^2}{b^2} = z^2$ in $\mathbb{C}\mathrm{P}^2$. $ABCD$ is a~parallelogram iff $ABCD$ is circumscribed \\ about $\gamma_\lambda: \cfrac{x^2}{a^2 + \lambda} + \cfrac{y^2}{b^2 + \lambda} = z^2$ with $\lambda~= -\cfrac{a^2b^2}{a^2 + b^2}$.}

Corentin Fierobe brilliantly studied the~connection between $n$-periodic orbits and corresponding to them values of $\lambda$ (see \cite{Fier-caustics}). In particular, in Section 7 of this article it is said that any complex 4-periodic orbit which do not have its edges on a~foci line is circumscribed about~$\gamma_\lambda$ with $\lambda~\in \{ -\cfrac{a^2b^2}{a^2 - b^2}, -\cfrac{a^2b^2}{a^2 + b^2}, \cfrac{a^2b^2}{a^2 - b^2} \}$. Moreover, any complex orbit inscribed in $\gamma$ and circumscribed about a~$\gamma_\lambda$ with $\lambda$ from the~list above is a~quadrilateral orbit.

Thus, our result is consistent with the~existing one. It is quite peculiar that one value out of three allows parallelogram orbits and parallelogram orbits only.

\vspace{10mm}

\textbf{Acknowledgment.} The author thanks Alexei Glutsyuk, Serge Tabachnikov, and Gil Bor for inspiration, valuable discussions, help, and advise.

\end{document}